\documentclass{amsart}
              [1997/02/02 v1.2e PROC Author Class]

\copyrightinfo{2006}%            % copyright year
  {American Mathematical Society}% copyright holder

\newtheorem{theorem}{Theorem}[section]

\newtheorem{proposition}[theorem]{Proposition}

\begin{document}

\title[Onepoint discontinuity set of separately continuous functions]{Onepoint discontinuity set of separately continuous functions on the product of two compact spaces}

\author{V.V.Mykhaylyuk}
\address{Department of Mathematics\\
Chernivtsi National University\\ str. Kotsjubyn'skogo 2,
Chernivtsi, 58012 Ukraine}
\email{vmykhaylyuk@ukr.net}

%    General info

%\date{January 1, 1994 and, in revised form, June 22, 1994.}

\commby{Ronald A. Fintushel}

%\dedicatory{This paper is dedicated to our authors.}


\begin{abstract}
It is investigated the existence of a separately continuous function $f:X\times Y\to
\mathbb R$ with an onepoint set of discontinuity for topological spaces $X$ and $Y$
which satisfy compactness type conditions. In particular, it is shown that for compact
spaces $X$ and $Y$ and nonizolated points $x_0\in X$ and $y_0\in Y$ there exists a
separately continuous function $f:X\times Y\to \mathbb R$ with the set $\{(x_0,y_0)\}$
of discontinuity points if and only if there exist sequences of nonempty functional
open sets which converge to $x_0$ and $y_0$ in $X$ and $Y$ respectively.
\end{abstract}

\maketitle
\section{Introduction}

It follows from Namioka Theorem [1] that the discontinuity points set of separately continuous function defined on the product of two compact spaces contains in the product of two meager sets. The question on a characterization of discontinuity points set of separately continuous functions defined on the product of two compacts naturally arises (see [2]). In particular, is it true that every onepoint set with non-isolated projections in the products of compact spaces is a discontinuity points set for some separately continuous function? It was shown in [3, Theorem 9] that this question has the negative answer. Namely, it was proved in [3] that does not exist separately continuous function defined on the product of two Tikhonoff cubes at least one of which has uncountable weight with onepoint set of discontinuity points. This result was generalized in [4].

Conditions on topological spaces $X$ and $Y$ and on points $a\in X$ and $b\in Y$ under which there exists a separately continuous function $f:X\times Y\to \mathbb R$ with the set $\{(a,b)\}$ of discontinuity points were studied in [5] (see also [6]). It was shown in [5] that such separately continuous function exists if the following conditions are valid: $X$ and $Y$ are completely regular; $a$ and $b$ are non-isolated $G_{\delta}$-points; $b$ has a countable base of neighborhoods in $Y$ or $Y$ is locally connected in the point $b$. It was constructed in [7] a separately continuous function, defined on the product of two Eberlein compacts, with the given set of discontinuity points. It follows from [7] that similar statement is true if $X$ and $Y$ are Eberlein compacts and $a\in X$ and $b\in Y$ are non-isolated points. Price-Symon property of Eberlein compacts (see [8, p.170]) is a main technical tool in the proof of results from [7]. Note that Tikhonoff cube with uncountable weight has not Price-Symon property. Therefore the following question naturally arises: is the existence of convergent sequences of open sets in compact spaces $X$ and $Y$ necessary for the existence of separately continuous function $f:X\times Y\to
\mathbb R$ with onepoint discontinuity set.

In the paper we study this question. Namely, for some classes of topological spaces $X$ and $Y$ we show that the following conditions are equivalent:

$(i)$ there exist sequences $(U_n)^{\infty}_{n=1}$ and $(V_n)^{\infty}_{n=1}$ of nonempty functionally open sets $U_n\subseteq X$ and $V_n\subseteq Y$, which converge to $x_0\in X$ and $y_0\in Y$ respectively, besides $x_0\not \in U_n$ and $y_0\not \in V_n$ for every $n\in \mathbb N$;

$(ii)$ there exists a separately continuous function $f:X\times Y\to \mathbb R$ for which set of discontinuity points set equals to $\{(x_0,y_0)\}$.

Firstly, using the dependence of function of some coordinates we prove this proposition if $X$ is a separable pseudocompact space and $Y$ is a compact, or $X$ is a pseudocompact with the countable Suslin number and $Y$ is a Valdivia compact. Further we obtain an analogical result for compact spaces $X$ and $Y$.

\section{Notions and auxiliary statements}

The set of discontinuity points for mapping $f:X\to Y$ defined on a topological space $X$ and valued in a topological space $Y$ we denote by $D(f)$. Moreover, if $Y$ is a metric space with a metric $d$ and $A\subseteq X$ is a nonempty set, then the real
$$\omega_f(A)=\sup\limits_{x',x''\in A}d(f(x'),f(x''))$$ is called by {\it the oscillation of $f$ on $A$}, and the real
$$\omega_f(x_0)=\inf\limits_{U\in {\mathcal U}}\omega_f(U),$$, where ${\mathcal U}$ is the system of all neighborhoods of $x_0\in X$, is called by {\it the oscillation of $f$ at $x_0$}.

For a function $f:X\to \mathbb R$ by ${\rm supp}\,f$ we denote the support $\{x\in X:f(x)\ne 0\}$ of $f$.

A set $A$ in a topological space $X$ is called {\it functionally open}, if there exists a continuous function $f:X\to [0,1]$ with $A=f^{-1}((0,1])$.

We say that a sequence $(A_n)_{n=1}^\infty$ of nonempty subsets $A_n$ of topological space $X$ {\it converges to a point $x_0\in X$} (denote by $A_n\to x_0$), if for every neighborhood $U$ of $x_0$ in $X$ there exists an integed $n_0$ such that $A_n\subseteq U$ foe all $n\ge n_0$.

A family $(A_i: i\in I)$ of subsets $A_i$ of a topological space $X$ is called {\it locally finite}, if for every $x\in X$ there exists a neighborhood $U$ of $x$ in $X$ such that the set $\{i\in I: U\cap A_i\ne \O\}$ is finite, and {\it pointwise finite}, if for every $x\in X$ the set $\{i\in I: x\in A_i\}$ is finite.

Completely regular space $X$ is called {\it pseudocompact}, if every continuous function $f:X\to \mathbb R$ is bounded, and {\it Lindeloff}, if for every open cover of $X$ there exists an at most countable subcover.

For infinite cardinal $\aleph$ we say that a topological space $X$ has $(II_{\aleph})$, if $|I|\leq \aleph$ for every pointwise finite family $(A_i: i\in I)$ of open in $X$ nonenpty sets $A_i$. This property was denoted by $(II_{\aleph})$ in [9], where it was used for investigation of dependence on some coordinates of separately continuous functions defined on the product of two space-products.

A topological space $X$ has {\it countable Suslin number}, if every system of nonempty pairwise disjoint  open in $X$ has at most countable cardinality.

\begin{proposition}\label{p:2.1}
A Baire space $X$ has $(II_{\aleph_0})$ if and only if $X$ has countable Suslin number.
\end{proposition}

\begin{proof} The necessity is obvius. We prove the sufficiency. Let $X$ has countable Suslin number and $(U_i:i\in I)$ is a pointwise finite family of open in $X$ nonempty sets $U_i$, besides $I$ is infinite. Since $X$ is Baire, for every $i\in I$ there exists an open in $X$ nonempty set $V_i\subseteq U_i$ such that the set $\{j\in I: U_j\cap V_i\ne \O \}$ is finite. Then the set $\{j\in I: V_j\cap V_i\ne \O\}$ is finite for every $i\in I$. Using a maximal set $J\subseteq I$ such that the family $(V_j: j\in J)$ consists of pairwise disjoint sets, we obtain $|I|\leq \aleph_0 \cdot |J|\leq \aleph^2_0 = \aleph_0$. Thus, $X$ has $(II_{\aleph_0})$.
\end{proof}

A continuous mapping $f:X\to Y$ defined on topological space $X$ and valued in topological space $Y$ is called {\it perfect}, if $f$ is closed that is the set $B=f(A)=\{f(a):a\in A\}$ is closed in $Y$ for every closed in $X$ set $A$, and the set $f^{-1}(y) = \{x\in X: f(x)=y\}$ is compact in $X$ for every $y\in Y$.

\begin{proposition}\label{p:2.2} Let $X$, $Y$ be topological spaces, $\varphi:Y\to Y_0$ be a perfect surjection, $f_0:X\times Y_0\to \mathbb R$ and $f:X\times Y\to \mathbb R$ such that $f(x,y)=f_0(x,\varphi(y))$ for every $x\in X$ and $y\in Y$. Then $D(f_0)=D$, where $D=\{(x,\varphi(y)):(x,y)\in D(f)\}$.
\end{proposition}

\begin{proof} Let $(x,y)\in D(f)$ and $y_0=\varphi(y)$. It is clear that $(x,y_0)\in D(f_0)$. Thus, $D\subseteq D(f_0)$.

Let $(x_0,y_0)\not \in D$. We put $K=\varphi^{-1}(y_0)$. Since $\varphi$ is perfect, $K$ is a compact set in $Y$. Note that $f$ is continuous at every point $(x_0,y)$, where $y\in K$, besides $f(x_0,y)=f_0(x_0,y_0)$. Therefore for every $\varepsilon >0$ there exists an open neighborhood $U$ of $x_0$ in $X$ and an open set $G$ in $Y$ such that $K\subseteq G$ and $|f(x,y)-f_0(x_0,y_0)|<\varepsilon$ for every $x\in U$ and $y\in G$. The set $Y\setminus G$ is closed in $Y$ and $\varphi$ is perfect, therefore the set $F=\varphi (Y\setminus G)$ is closed in $Y_0$, besides $y_0\not \in F$. Put $V_0=Y_0\setminus F$. Clearly that $V_0$ is an neighborhood of $y_0$ and $\varphi^{-1}(V_0)\subseteq G$. Let $x\in U$ and $y'\in V_0$. We choose $y''\in G$ such that $\varphi(y'')=y'$. Then $|f_0(x,y')-f_0(x_0,y_0)| = |f(x,y'')-f(x_0,y_0)|<\varepsilon$. Thus, $f_0$ is continuous at $(x_0,y_0)$.
Hence, $D(f_0)\subseteq D$.
\end{proof}

The following proposition proves the implication $(i)\Longrightarrow(ii)$.

\begin{proposition}\label{p:2.3} Let $X$, $Y$ be topological spaces, $x_0\in X$, $y_0\in Y$, $(U_n)^{\infty}_{n=1}$ and $(V_n)^{\infty}_{n=1}$ be sequences of nonempty functionally open in $X$ and $Y$ sets $U_n\subseteq X$ and $V_n\subseteq Y$ respectively such that $U_n\to x_0$ and $V_n\to y_0$, besides $x_0\not \in U_n$ and $y_0\not \in V_n$ for every $n\in \mathbb N$. Then there exists a separately continuous function $f:X\times Y \to \mathbb R$ such that $D(f)=\{(x_0,y_0)\}$.
\end{proposition}

\begin{proof} Let $\varphi_n:X\to [0,1]$ and $\psi_n:Y\to [0,1]$ be continuous functions such that $U_n=\varphi^{-1}_n((0,1])$, $V_n=\psi^{-1}_n((0,1])$ and $\sup\limits_{x\in X} \varphi_n(x) = \sup\limits_{y\in Y} \psi_n(y)= 1$ for every $n\in \mathbb N$. It easy to see that the function $f:X\times Y \to \mathbb R$, $f(x,y)= \sum\limits^{\infty}_{n=1}\varphi_n(x)\cdot \psi_n(y)$, is required.
\end{proof}

\section{Dependence on $\aleph_0$ coordinates}

Let $Z$, $T$ be sets, $Y\subseteq {\mathbb R}^T$ and $f:Y\to Z$. We say that {\it $f$ concentrated on $S$}, where $S\subseteq T$, if $f(y')=f(y'')$ for every $y', y''\in Y$ with $y'|_S = y''|_S$. Moreover, if the cardinality $|S|$ of $S$ is $\leq\aleph_0$, then we say that {\it $f$ depends on $\aleph_0$ coordinates}.

Moreover, let $X$ be a set and $g:X\times Y\to Z$. Then {\it $g$ concentrated on $S\subseteq T$ with respect to the second variable}, if $f(x,y')=f(x,y'')$
for every $x\in X$ and $y',y''\in Y$ with $y'|_S = y''|_S$, and {\it $g$ depends on $\aleph_0$ coordinates with respect to the second variable}, if $|S|\leq \aleph_0$ for some such set $S$.

\begin{theorem}\label{th:3.1}
Let $X$ be a separable topological space and $Y\subseteq {\mathbb R}^T$ be a Lindeloff space. Then every separately continuous function $f:X\times Y\to \mathbb R$ depends on $\aleph_0$ coordinates with respect to the second variable.
\end{theorem}

\begin{proof} Since $Y$ is Lindeloff, every continuous function $g:Y\to \mathbb R$ depends on $\aleph_0$ coordinates. Therefore for every $x\in X$ there exists an at most countable set $T_x\subseteq T$ such that $f(x,y')=f(x,y'')$ for every $y',y''\in Y$ with $y'|_{T_x}=y''|_{T_x}$. Let $A$ be a countable everywhere dense in $X$ set. We put $S=\bigcup\limits_{a\in A}T_a$. Clearly that $|S|\leq \aleph_0$. For every $y',y''\in Y$ with $y'|_S=y''|_S$ we have $y'|_{T_a}=y''|_{T_a}$. Therefore, $f(a,y')=f(a,y'')$ for every $a\in A$. Taking into account that $f$ is continuous with respect to $x$ and the closure $\overline {A}$ of $A$ coincides with $X$ we obtain that $f(x,y')=f(x,y'')$ for all $x\in X$.
\end{proof}

\begin{theorem}\label{th:3.2} Let $X$ be a topological space with $(II_{\aleph_0})$ and $Y\subseteq {\mathbb R}^T$ be a compact, besides $Y=\overline{B}$, where $B=\{y\in Y: |{\rm supp}\,y|\leq \aleph_0\}$. Then for every separately continuous function $f:X\times Y\to\mathbb R$ depends on $\aleph_0$ coordinates with respect to the second variable.
\end{theorem}

\begin{proof} Firstly we prove that for every $\varepsilon >0$ there exists an at most countable set $S_{\varepsilon}\subseteq T$ such that $|f(x,b')-f(x,b'')|\leq \varepsilon$ for every $x\in X$ and $b',b''\in B$ with $b'|_{S_{\varepsilon}}=b''|_{S_{\varepsilon}}$.

Assume the contrary. That is there exists $\varepsilon >0$ such that for every at most countable set $S\subseteq T$ there exist $x\in X$ and $b',b''\in B$ such that $b'|_S=b''|_S$ and $|f(x,b')-f(x,b'')|>\varepsilon$. Using transfinite induction we construct families $(S_{\alpha}:\alpha<{\omega}_1)$ of at most countable sets $S_{\alpha}\subseteq T$, $(b_{\alpha}:\alpha<{\omega}_1)$, $(c_{\alpha}:\alpha<{\omega}_1)$ and $(x_{\alpha}:\alpha<{\omega}_1)$ of points $b_{\alpha},
c_{\alpha}\in B$ and $x_{\alpha}\in X$ such that

$a$)\,\,\,$b_{\alpha}|_{S_{\alpha}}=c_{\alpha}|_{S_{\alpha}}$ for every $\alpha<\omega_1$;

$b$)\,\,\,$S_{\alpha}\subseteq S_{\beta}$ for every $\alpha<\beta<\omega_1$;

$c$)\,\,\,${\rm supp}\,b_{\alpha}\subseteq S_{\alpha+1}$, ${\rm
supp}\,c_{\alpha}\subseteq S_{\alpha+1}$ for every $\alpha<\omega_1$;

$d$)\,\,\,$|f(x_{\alpha}, b_{\alpha})-f(x_{\alpha}, c_{\alpha})|>\varepsilon$ for every $\alpha <\omega_1$.

Take an at most countable set $S_1\subseteq T$. According to the assumption there exist $x_1\in X$ and $b_1, c_1\in B$ such that $b_1|_{S_1}=c_1|_{S_1}$
and $|f(x_1,b_1)-f(x_1,c_1)|>\varepsilon$. Put $S_2=S_1\cup{\rm supp}\,b_1\cup{\rm supp}\, c_1$. Clearly that $|S_2|\leq \aleph_0$. We choose $x_2\in X$ and
$b_2,c_2\in B$ such that $b_2|_{S_2}=c_2|_{S_2}$ and $|f(x_2,b_2)-f(x_2,c_2)|>\varepsilon$.

Suppose that for some $\beta<\omega_1$ the famplies $(S_{\alpha}:\alpha<\beta)$, $(b_{\alpha}:\alpha<\beta)$, $(c_{\alpha}:\alpha<\beta)$ and $(x_{\alpha}:\alpha<\beta)$ are cunstructed. Put $S_{\beta}=\bigcup\limits_{\alpha<\beta}(S_{\alpha}\cup{\rm supp}\,b_{\alpha}\cup{\rm supp}\,c_{\alpha})$. Since for $\alpha<\beta$ all sets $S_{\alpha}$, ${\rm supp}\,b_{\alpha}$ and ${\rm supp}\,c_{\alpha}$ are at most countable, $|S_{\beta}|\leq \aleph_0$. Now using our assumption we choose $x_{\beta}$ and $b_{\beta}, c_{\beta}\in B$ such that $b_{\beta}|_{S_{\beta}}=c_{\beta}|_{S_{\beta}}$ and
$|f(x_{\beta},b_{\beta})-f(x_{\beta},c_{\beta})|>\varepsilon$.

Using the continuity of $f$ with respect to $x$ and the condition $d$) for every $\alpha<\omega_1$ we choose an open neighborhood $U_{\alpha}$ of $x_{\alpha}$ in $X$ such that $|f(x,b_{\alpha})-f(x,c_{\alpha})|>\varepsilon$ for every $x\in U_{\alpha}$. Since $X$ has $(II_{\aleph_0})$, the family $(U_{\alpha}:\alpha<\omega_1)$ is not pointwise finite. Thus, there exists $x_0\in X$ and a strictly increasing sequence $(\alpha_n)^{\infty}_{n=1}$ of at most countable ordinals $\alpha_n$ such that $|f(x_0,b_{\alpha_n})-f(x_0,c_{\alpha_n})|>\varepsilon$ for every $n\in \mathbb N$.

We put $T_n=S_{\alpha_n}$, $v_n=b_{\alpha_n}$ and $w_n=c_{\alpha_n}$ for $n\in \mathbb N$. Using the compactness of $Y$ and the continuity of $f^{x_0}:Y\to \mathbb R$, $f^{x_0}(y)=f(x_0,y)$, we choose a finite set $T_0\subseteq T$ such that $|f(x_0,y')-f(x_0,y'')|<\varepsilon$ for every $y',y''\in Y$ with  $y'|_{T_0}=y''|_{T_0}$. Since $|f(x_0,v_n)-f(x_0,w_n)|>\varepsilon$, $v_n|_{T_0}\ne w_n|_{T_0}$. But according to $a$), $v_n|_{T_n}= w_n|_{T_n}$ and according to $b$) and $c$), the functions $v_n|_{T\setminus T_{n+1}}$ and $w_n|_{T\setminus T_{n+1}}$ are null-functions. Therefore, $v_n|_{T\setminus T_{n+1}}=w_n|_{T\setminus T_{n+1}}$. Thus, $T_0\cap(T_{n+1}\setminus T_n)\ne\O$ for every $n\in \mathbb N$. Taking into account that the sequence $(T_n)^{\infty}_{n=1}$ is increasing, we obtain that the set $T_0$ is infinite. But this contradicts to the choice of $T_0$. Hence, the existence of $S_{\varepsilon}$ is proved.

Put $S_0=\bigcup\limits_{n=1}^{\infty}S_{\frac{1}{n}}$. Clearly that $f(x,b')=f(x,b'')$ for every $x\in X$, $b',b''\in B$ with $b'|_{S_0}=b''|_{S_0}$.
Fix any points $x\in X$, $y',y''\in Y$ such that $y'|_{S_0}=y''|_{S_0}$. Using the continuity of $f^x:Y\to\mathbb R$, $f^x(y)=f(x,y)$, on compact space $Y\subseteq {\mathbb R}^T$ we choose an at most countable set $T_0\subseteq T$ such that $f(x,y_1)=f(x,y_2)$ for every $y_1,y_2\in Y$ with $y_1|_{T_0}=y_2|_{T_0}$. Since $B$ is countably compact set, $Y=\overline{B}$ and the function $f^x:Y\to\mathbb R$, $f^x(y)=f(x,y)$, is continuous, there exist $b',b''\in B$ such that $b'|_{T_0\cup S_0}=y'|_{T_0\cup S_0}$, $f(x,y')=f(x,b')$, $b''|_{T_0\cup S_0}=y''|_{T_0\cup S_0}$ and $f(x,y'')=f(x,b'')$. Then
$f(x,y')=f(x,b')=f(x,b'')=f(x,y'')$. Thus, $f$ concetrated on $S_0$, therefore $f$ depends on at most countable coordinates with respect to the second variable.
\end{proof}

\section{Onepoint discontinuity set of separately continuous functions}

\begin{proposition}\label{p:4.1} Let $X$ be a pseudocompact space, $Y$ be a topological space, $x_0\in X$, $y_0\in Y$, $f:X\times Y\to \mathbb R$ with $D(f)=\{(x_0,y_0)\}$, $\delta>0$ and $U_0$ be a closed neighborhood of $x_0$ in $X$ such that $|f(x,y_0)-f(x_0,y_0)|<\delta$ for every $x\in U_0$, $(U_n)^{\infty}_{n=1}$ and $(V_n)^{\infty}_{n=1}$ be sequences of open nonempty sets $U_n\subseteq U_0$ and $V_n$ in $X$ and $Y$ respectively such that $|f(x,y)-f(x_0,y_0)|>\delta$ for all $(x,y)\in U_n\times V_n$ by $n\in \mathbb N$. Then if $(V_n)^{\infty}_{n=1}$ converges to $y_0$, then $(U_n)^{\infty}_{n=1}$ converges to $x_0$.
\end{proposition}

\begin{proof} Let $U$ be a closed neighborhood of $x_0$ in $X$. Suppose that the set $N=\{n\in \mathbb N: U_n\setminus U \ne \O\}$ infinite. Without loss of generality we can propose that $N=\mathbb N$. Put $\tilde{U}_n=U_n\setminus U$ for every $n\in \mathbb N$. Since $X$ pseudocompact, according to [10, p. 311] the family $(\tilde{U}_n:n\in \mathbb N)$ of open nonempty sets $\tilde{U}_n$ is not locally finite in $X$. Therefore there exists $\tilde{x}\in U_0$ such that an arbitrary neighborhood $\tilde{U}$ of $\tilde{x}$ in $X$ intersects with infinite quantity of elements of the family $(\tilde{U}_n:n\in \mathbb N)$. Since $V_n\to y_0$, every neighborhood $W$ of $(\tilde{x},y_0)$ in $X\times Y$ intersects with infinite quantity of elements of the family $(W_n:n\in\mathbb N)$, where $W_n=\tilde{U}_n\times V_n$. Note that $\tilde{x}\ne x_0$. Therefore the function $f$ is continuous at $(\tilde{x},y_0)$. Taking into account that $|f(x,y)-f(x_0,y_0)|>\delta$ for every $(x,y)\in W_n$ by $n\in\mathbb N$, we obtain that $|f(\tilde{x},y_0)-f(x_0,y_0)|\geq \delta$. But this contradicts to $\tilde{x}\in U_0$. Thus, $N$ is finite and $U_n\to x_0$.
\end{proof}

A compact space $Y$ is {\it compact Valdivia}, if $Y$ is homeomorphic to a compact space $Z\subseteq {\mathbb R}^T$ such that the set $\{z\in
Z:|{\rm supp}\,z|\leq \aleph_0\}$ is dense in $Z$.

The following theorem is a main result of this section.

\begin{theorem}\label{th:4.2} Let $X$ be a separable pseudocompact space and $Y$ be a compact space or let $X$ be a pseudocompact space with countable Suslin number and $Y$ be a compact Valdivia, $x_0\in X$, $y_0\in Y$ be nonisolated points in corresponding spaces. Then the following conditions are equivalent:

$(i)$\,\,\,there exist sequences $(U_n)^{\infty}_{n=1}$ and $(V_n)^{\infty}_{n=1}$ of nonempty open sets $U_n\subseteq X$ and $V_n\subseteq Y$ which converges to $x_0$ and $y_0$ respectively;

$(ii)$\,\,\,there exists a separately continuous function $f:X\times Y\to \mathbb R$ with $D(f)=\{(x_0,y_0)\}$.
\end{theorem}

\begin{proof} Since $X$ and $Y$ are completely regular, the implication $(i)\Longrightarrow(ii)$ follows from Proposition \ref{p:2.3}.

$(ii)\Longrightarrow(i)$. Let $X$ be a separable pseudocompact space, $Y$ be a compact and $f:X\times Y\to \mathbb R$ be a separately continuous function with $D(f)=\{(x_0,y_0)\}$. Without loss of generality we can propose that $Y\subseteq {\mathbb R}^T$, where $T$ is a set. According to Theorem \ref{th:3.1}, the function $f$ depends on $\aleph_0$ coordinates with respect to the second variable. Hence there exists an at most countable set $S\subseteq T$ such that  $f(x,y')=f(x,y'')$ for every $x\in X$ and $y',y''\in Y$ with $y'|_S=y''|_S$.

We put $\varphi:Y\to {\mathbb R}^S$, $\varphi(y)=y|_S$, $Z=\varphi(Y)$, $f_0:X\times Z\to\mathbb R$, $f_0(x,\varphi(y))=f(x,y)$, where $y\in Y$. Clearly that $\varphi:Y\to Z$ is perfect. Therefore according to Proposition \ref{p:2.2}, $D(f_0)=\{(x_0,z_0)\}$, where $z_0=\varphi(y_0)$. Moreover, $f_0$ is separately continuous (the continuity with respect to the first variable immediately follows from the continuity of $f$ with respect to the first variable, and the continuity with respect to the second variable can be obtained using Proposition \ref{p:2.2} with the first multiplier $\{x\}$). We choose
$\delta>0$ with $\omega_{f_0}(x_0,z_0)>3\delta$, and choose closed neighborhoods $U_0$ and $W_0$ of $x_0$ and $z_0$ in $X$ and $Z$ respectively such that $|f_0(x,z_0)-f_0(x_0,z_0)|<\delta$ and $|f_0(x_0,z)-f_0(x_0,z_0)|<\delta$ for every $x\in U_0$ and $z\in W_0$.

Note that $Z$ is a metrizable compact. Fix any base $(G_n)^{\infty}_{n=1}$ of open neighborhoods $G_n\subseteq W_0$ of $z_0$ in $Z$ and choose an arbitrary open neighborhood $\tilde{U}$ of $x_0$ in $X$. Since $\omega_{f_0}(\tilde{U}\times G_n)>3\delta$ and $f_0$ is continuous at all points except $(x_0,z_0)$, there exist sequences $(U_n)^{\infty}_{n=1}$ and $(W_n)^{\infty}_{n=1}$ of nonempty open sets $U_n\subseteq \tilde {U}$ and $W_n\subseteq W_0$ in $X$ and $Z$ respectively such that $|f_0(x,z)-f_0(x_0,z_0)|>\delta$ for all $(x,z)\in U_n\times W_n$ by $n\in\mathbb N$. Clearly, that $W_n\to z_0$. Therefore according to Proposition \ref{p:4.1}, $U_n\to x_0$.

Put $V_n={\varphi}^{-1}(W_n)$ by $n=0,1,2,\dots$. Note that $|f(x_0,y)-f(x_0,y_0)|<\delta$ for every $y\in V_0$ and $|f(x,y)-f(x_0,y_0)|>\delta$ for all $(x,y)\in U_n\times V_n$ by $n\in \mathbb N$. Interchanging variables and using Proposition \ref{p:4.1}, we obtain that $V_n\to y_0$.

Now let $X$ is a pseudocompact with countable Suslin number and $Y$ is a compact Valdivia. We can propose that $Y\subseteq {\mathbb R}^T$, besides the set $\{y\in Y:|{\rm supp}\,y|\leq\aleph_0\}$ is dense in $Y$. It follows from [10, p.~311] that $X$ is a Baire space. Therefore Proposition \ref{p:2.1} and Theorem \ref{th:3.2} imply that every separately continuous function $f:X\times Y\to\mathbb R$ depends on $\aleph_0$ coordinates with respect to the second variable. Further we reason analogously as in the previous case.
\end{proof}

\section{Product of compact spaces}

\begin{theorem} \label{th:5.1} Let $X$, $Y$ be a compact spaces and $x_0\in X$, $y_0\in Y$ be a nonisolated points in corresponding spaces. Then the following conditions are equivalent:

$(i)$\,\,\,there exist sequences $(U_n)^{\infty}_{n=1}$ and $(V_n)^{\infty}_{n=1}$ of nonempty functionally open sets $U_n\subseteq X$ and $V_n\subseteq Y$ which converges to $x_0$ and $y_0$ respectively, besides $x_0\not\in U_n$ and $y_0\not\in V_n$ by $n\in \mathbb N$;

$(ii)$\,\,\,there exists a separately continuous function $f:X\times Y\to \mathbb R$ with $D(f)=\{(x_0,y_0)\}$.
\end{theorem}

\begin{proof} As in the proof of Theorem \ref{th:4.2} it is sufficient to verify the implication $(ii)\Longrightarrow(i)$. Let $f:X\times Y\to \mathbb R$ is a separately continuous function with $D(f)=\{(x_0,y_0)\}$. We consider the continuous mapping $\varphi:X\to C_p(Y)$, $\varphi(x)(y)=f(x,y)$. Put $\tilde{X}=\varphi(X)$, $g:\tilde{X}\times Y\to \mathbb R$, $g(\tilde{x},y)=\tilde{x}(y)=f(x,y)$, where $\tilde{x}=\varphi(x)$. Clearly that $g$ is a separately continuous function. Since $X$ is a compact space, $\varphi$ is a perfect mapping. According to Proposition \ref{p:2.2}, $D(g)=\{(\tilde{x}_0,y_0)\}$, where $\tilde{x}_0=\varphi(x_0)$. Let $\tilde{x}\in A=\tilde{X}\setminus\{\tilde{x}_0\}$. Taking into account that $Y$ is a compact space and $g$ is continuous at every point of the set $\{\tilde{x}\}\times Y$, we obtain that for every $\varepsilon >0$ there exists a neighborhood $\tilde{U}$ of $\tilde{x}$ in $\tilde{X}$
such that $|\tilde{x}(y)-\tilde{u}(y)|<\varepsilon$ for every $y\in Y$ and $\tilde{u}\in \tilde{U}$. Thus, on the set $A$ the topology of pointwise convergence coincides with the normed topology of Banach space $C(X)$. Therefore, $A$ is a metric subspace of $\tilde{X}$.

Now we consider the mapping $\psi:Y\to C_p(\tilde{X})$, $\psi(y)(\tilde{x})=g(\tilde{x},y)$. Put $\tilde{Y}=\psi(Y)$, $h:\tilde{X}\times\tilde{Y}\to \mathbb R$, $h(\tilde{x},\tilde{y})=g(\tilde{x},y)$, where $\tilde{y}=\psi(y)$. Analogously as in the previous reasoning we have $D(h)=\{(\tilde{x}_0,\tilde{y}_0)\}$, where $\tilde{y}_0=\psi(y_0)$, and the set $B=\tilde{Y}\setminus \{\tilde{y}_0\}$ is a metric subspace of $\tilde{Y}$.

Take $\delta>0$ such that $\omega_h(\tilde{x}_0,\tilde{y}_0)> 4\delta$, and choose closed neighborhoods $\tilde{U}_0$ and $\tilde{V}_0$ of $\tilde{x}_0$ and $\tilde{y}_0$ in $\tilde{X}$ and $\tilde{Y}$ respectively such that $|h(\tilde{x},\tilde{y}_0)-h(\tilde{x}_0,\tilde{y}_0)|<\delta$ and $|h(\tilde{x}_0,\tilde{y})-h(\tilde{x}_0,\tilde{y}_0)|<\delta$ for every $\tilde{x}\in \tilde{U}_0$ ³ $\tilde{y}\in \tilde{V}_0$. Put $Z=\tilde{X}\times
\tilde{Y}$, $z_0=(\tilde{x}_0,\tilde{y}_0)$ and $W_0={\rm int}(\tilde{U}_0)\times {\rm int}(\tilde{V}_0)$, where by ${\rm int}(C)$ we denote the interior of $C$ in corresponding space. For every point $z\in A\times B$ we choose an open neighborhood $G_z$ of $z$ in $Z$ such that $z_0\not \in \overline{G}_z$ and $\omega_h(G_z)< \delta$. Since $A\times B$ is metrizable subspace of $Z$, according to Stone's Theorem on paracompactness of metrizable space [10, p.414], from the open cover $(G_z: z\in A\times B)$ of space $A\times B$ we can choose an locally finite open subcover $(W_i:i\in I)$. Put $J=\{i\in I: W_i\cap
W_0\ne \O \,\,\,\mbox{and}\,\,\,|h(z)-h(z_0)|>2\delta\,\,\,\mbox{for\,\,\,some}\,\,\,z\in W_i\}$. Since $\omega_h(z_0)>4\delta$, for every neighborhood $W\subseteq W_0$ of $z_0$ there exists $z\in W$ such that $|h(z)-h(z_0)|>2\delta$, besides according to the choice $\tilde{U}_0$ and $\tilde{V}_0$ the point $z$ belongs to the set $A\times B$. Therefore $z_0\in \overline{\bigcup\limits_{i\in J}W_i}$. Taking into account that $z_0\not \in \overline{W}_i$ for every $i\in I$ we obtain that the set $J$ is infinite. Moreover, note that since $\omega_h(W_i)<\delta$ for $i\in I$, $|h(z)-h(z_0)|>\delta$ for every $j\in J$ and $z\in W_j$. Since $|h(z)-h(z_0)|<\delta$ for every $z\in ((\{\tilde{x}_0\}\times \tilde{V}_0)\cup(\tilde{U}_0\times\{\tilde{y}_0\}))\setminus \{z_0\} = C$ and the function $h$ is continuous at every point of set $C$, $\overline{\bigcup\limits_{i\in J}W_i}\bigcap C=\O$.

We choose an arbitrary countable set $\{j_1, j_2,\dots\}\subseteq J$ and put $\tilde{W}_n=W_{j_n}\cap W_0$ for $n\in \mathbb N$. It follows from the definition of set $J$ that all sets $\tilde{W}_n$ are nonempty. Note that the family $(\tilde{W}_n:n\in \mathbb N)$ is locally finite at every point of set $(A\times B)\cup(Z\setminus (\tilde{U}_0\times \tilde{V}_0))\cup C=Z\setminus \{z_0\}$. Let $W$ be an arbitrary closed neighborhood of $z_0$ in $Z$. Then the family $(\tilde{W}_n\setminus W: n\in \mathbb N)$ is a locally finite family of open sets in the compact $Z$. Therefore $\tilde{W}_n\setminus W\ne \O$ only for finite quantity of integers $n$. Hence, there exists $n_0\in \mathbb N$ such that $\tilde{W}_n\subseteq W$ for all $n\geq n_0$. Thus, $\tilde{W}_n\to z_0$.

For every $n\in \mathbb N$ we choose open sets $\tilde{U}_n$ and $\tilde{V}_n$ in $\tilde{X}$ and $\tilde{Y}$ respectively such that $\tilde{U}_n\times\tilde{V}_n\subseteq\tilde{W}_n$. Clearly that $\tilde{U}_n\to\tilde{x}_0$ and $\tilde{V}_n\to\tilde{y}_0$. For $n=0,1,2,\dots$
we put $U_n=\varphi^{-1}(\tilde{U}_n)$ and $V_n=\psi^{-1}(\tilde{V}_n)$. The sets $U_0$ and $V_0$ are closed, and sets $U_n$ and $V_n$ for $n\in\mathbb N$ are functionally open in $X$ and $Y$ respectively. Now using Proposition \ref{p:4.1} for the function $g$, we obtain that the convergence $\tilde{U}_n\to\tilde{x}_0$ implies the convergence $V_n\to y_0$. Using similar arguments to the function $f$, we obtain that $U_n\to x_0$.
\end{proof}

\bibliographystyle{amsplain}

\end{document}